\documentclass[11pt, reqno]{amsart}

\usepackage[T1]{fontenc}
\usepackage{babel}
\usepackage{lmodern}

\usepackage{latexsym}
\usepackage{amssymb}
\usepackage{mathrsfs}
\usepackage[leqno]{amsmath}
\usepackage{fancybox,color}
\usepackage{enumerate}
\usepackage[latin1]{inputenc}
\usepackage[active]{srcltx}
\usepackage[colorlinks]{hyperref}
\usepackage{hyperref}
\usepackage{eurosym}
\usepackage{url}
\usepackage{xfrac}
\usepackage{color}

\usepackage{fancyhdr}
\usepackage{xfrac}
\usepackage{xcolor}

\usepackage[foot]{amsaddr}

\newcommand{\kom}[1]{}
\renewcommand{\kom}[1]{{\bf [#1]}}

\def\1{\raisebox{2pt}{\rm{$\chi$}}}

\newtheorem{theorem}{Theorem}[section]

\newtheorem{lemma}[theorem]{Lemma}
\newtheorem{proposition}[theorem]{Proposition}
\newtheorem{definition}[theorem]{Definition}
\newtheorem{remark}[theorem]{Remark}

\newcommand{\R}{{\mathbb R}}

\newcommand{\eps}{{\varepsilon}}
\def\1{\raisebox{2pt}{\rm{$\chi$}}}

\let\originalleft\left
\let\originalright\right
\renewcommand{\left}{\mathopen{}\mathclose\bgroup\originalleft}
\renewcommand{\right}{\aftergroup\egroup\originalright}

\newcommand{\abs}[1]{\left|#1\right|}

\newcommand{\Rn}{\mathbb{R}^n}

\newcommand{\aveint}[2]{\mathchoice%
	{\mathop{\kern 0.2em\vrule width 0.6em height 0.69678ex depth -0.58065ex
			\kern -0.8em \intop}\nolimits_{\kern -0.45em#1}^{#2}}%
	{\mathop{\kern 0.1em\vrule width 0.5em height 0.69678ex depth -0.60387ex
			\kern -0.6em \intop}\nolimits_{#1}^{#2}}%
	{\mathop{\kern 0.1em\vrule width 0.5em height 0.69678ex depth -0.60387ex
			\kern -0.6em \intop}\nolimits_{#1}^{#2}}%
	{\mathop{\kern 0.1em\vrule width 0.5em height 0.69678ex depth -0.60387ex
			\kern -0.6em \intop}\nolimits_{#1}^{#2}}}

\newcommand{\Om}{\Omega}

\newcommand{\vp}{\varphi}

\renewcommand{\div}{\operatorname{div}}

\newtheoremstyle{case}{3mm}{-1,5mm}{}{}{}{:}{ }{}
\theoremstyle{case}

\newcommand{\numberthis}{\addtocounter{equation}{1}\tag{\theequation}}

\makeatletter
\newcommand{\leqnomode}{\tagsleft@true}
\newcommand{\reqnomode}{\tagsleft@false}
\makeatother

\DeclareMathOperator{\Tr}{Tr}

\numberwithin{equation}{section}

\usepackage{etoolbox}

\makeatletter
\patchcmd{\@maketitle}
{\if@titlepage \newpage \else}
{\if@titlepage
	\vspace{\baselineskip}
	\else}
{}{}
\makeatother

\let\oldtocsection=\tocsection

\let\oldtocsubsection=\tocsubsection

\let\oldtocsubsubsection=\tocsubsubsection

\renewcommand{\tocsection}[2]{\hspace{0em}\oldtocsection{#1}{#2}}
\renewcommand{\tocsubsection}[2]{\hspace{2em}\oldtocsubsection{#1}{#2}}
\renewcommand{\tocsubsubsection}[2]{\hspace{4em}\oldtocsubsubsection{#1}{#2}}

\calclayout

\title{Elliptic Harnack's inequality for a singular nonlinear parabolic equation in non-divergence form}

\author{Tapio Kurkinen}
\email{tapio.j.kurkinen@jyu.fi}
\author{Mikko Parviainen}
\email{mikko.j.parviainen@jyu.fi}
\author{Jarkko Siltakoski}
\email{jarkko.j.m.siltakoski@jyu.fi}
\address{Department of Mathematics and Statistics
	University of Jyv{\"{a}}skyl{\"{a}}
	PO Box 35, FI-40014 Jyv{\"{a}}skyl{\"{a}}, Finland}

\date{\today}


\begin{document}

	\begin{abstract}
		We prove an elliptic Harnack's inequality for a general form of a parabolic equation that generalizes both the standard parabolic $p$-Laplace equation and the normalized version that has been proposed in stochastic game theory. This version of the inequality doesn't require the intrinsic waiting time and we get the estimate with the same time level on both sides of the inequality.
	\end{abstract}
	\maketitle
	
	\section{Introduction}
	In his monograph, DiBenedetto \cite[Theorem VII.1.2]{Dibenedetto1993} proved elliptic Harnack's inequality for the divergence form $p$-parabolic equation  in the supercritical case. In this case, the intrinsic waiting time required for degenerate parabolic equations is no longer needed. Instead he established Harnack's inequality with the same time level on both sides of the estimate akin to the elliptic case.
	
	In this paper, we prove elliptic Harnack's inequality for the following general non-divergence form version of the nonlinear parabolic equation
	\begin{equation}
	\label{eq:rgnppar}
	\partial_t u=\abs{\nabla u}^{q-p}\div\left(\abs{\nabla u}^{p-2}\nabla u\right)=\abs{\nabla u}^{q-2}(\Delta u + (p-2)\Delta_\infty^Nu),
	\end{equation}
	for a natural range of exponents. When $q = 2$, we get the normalized $p$-parabolic equation arising from the game theory, and when $q = p$, it is the standard $p$-parabolic equation. 	
	
	Elliptic Harnack's inequality, Theorem \ref{thm:ellipticharnack}, states that a non-negative solution satisfies the following local a priori estimate
	\begin{align*}
	\gamma^{-1}\sup_{B_r(x_0)}u(\cdot,t_0)\leq u(x_0,t_0)\leq\gamma\inf_{B_{r}(x_0)}u(\cdot,t_0).
	\end{align*}	
	DiBenedetto's proof uses the theory of weak solutions. Since the equation is in a non-divergence form, unless $q = p$, the usual weak theory based on integration by parts is not available in our case.
	Our proof uses the parabolic (forward) Harnack's inequality proven by Parviainen and Vázquez \cite{Parviainen2020} to estimate the solution in the past, constructing an explicit supersolution with infinite boundary values and using the comparison principle to get an estimate at our original time level. The idea both in the proof of the forward Harnack as well as in the derivation of the explicit supersolution is based on an equivalence result. Heuristically speaking, radial solutions to the original non-divergence form problem can be interpreted as solutions to the divergence form $p$-parabolic equation, but in a fictitious space dimension $d$.

	Nash discussed the possibility of elliptic Harnack's inequality for a parabolic equation in \cite{Nash1958}. Later Moser \cite{Moser1964} pointed out that such an estimate does not hold for the heat equation. For the $p$-parabolic equation elliptic Harnack's inequality is obviously false if $p>2$, and holds for $\frac{2n}{n+1}<p<2$. In addition to \cite{Dibenedetto1993}, Harnack's inequalities in the singular range have been studied for example by Dibenedetto, Gianazza and Vespri in \cite{Dibenedetto2009,Dibenedetto2010} and \cite{Dibenedetto2012}. 	The intrinsic forward Harnack's inequality for weak solutions of the $p$-parabolic equation was proven by Dibenedetto in \cite{Dibenedetto1988} and \cite{Dibenedetto1992}, see also \cite{Dibenedetto1993}, and later for equations with growth of order $p$ by Dibenedetto, Gianazza and Vespri in \cite{Dibenedetto2008} and by Kuusi in \cite{Kuusi2008}. For non-divergence form equations parabolic Harnack's inequalities and related H\"{o}lder regularity under additional restrictions were studied by Cordes \cite{Cordes1956} and Landis \cite{Landis1971}. With bounded and measurable coefficients parabolic Harnack's inequality was established  by Krylov and Safonov \cite{Krylov1981}.

	Since the equation \eqref{eq:rgnppar} is in non-divergence form except in a special case, the solutions in this paper are understood in the viscosity sense. 
	The suitable concept of viscosity solutions to the general equations (\ref{eq:rgnppar}) was established by Ohnuma and Sato \cite{Ohnuma1997}. In the special case $q=2$, we get the normalized $p$-parabolic equation that arises from the stochastic game theory \cite{Manfredi2010}. This non-divergence form special case as well as the general equation (\ref{eq:rgnppar}) have recently received attention in the works of Jin-Silvestre \cite{Jin2017}, Imbert-Jin-Silvestre \cite{Imbert2019}, H{\o}eg-Lindqvist \cite{Hoeg2019}, and Dong-Fa-Zhang-Zhou \cite{Dong2020} in addition to \cite{Parviainen2020}.
	
	\section{Main Results}
	Denote
	\begin{equation}
	\Delta_p^qu:=\abs{\nabla u}^{q-p}\div\left(\abs{\nabla u}^{p-2}\nabla u\right)=\abs{\nabla u}^{q-2}(\Delta u + (p-2)\Delta_\infty^Nu),
	\end{equation}
	where $p>1$ and $q>1$ are real parameters and
	\begin{equation*}
	\Delta_\infty^Nu=\sum_{i,j=1}^{n}\frac{\partial_{x_i}u \, \partial_{x_j}u \, \partial_{x_i x_j}u}{\abs{\nabla u}^2}
	\end{equation*}
	so the equation \eqref{eq:rgnppar} gets the form $\partial_tu=\Delta_p^qu$. Because the dimensions of the sets play part in some of the estimates we shall denote
	\begin{align*}
	Q_{r}^{-}(\theta)&=B_r(0)\times(-\theta r^q,0],\\
	Q_{r}^{+}(\theta)&=B_r(0)\times(0,\theta r^q)
	\end{align*}
	where $\theta$ is a positive parameter that determines the time-wise length of the cylinder relative to $r^q$. We denote the union of these sets as
	\begin{equation*}
	Q_{r}(\theta)=Q_{r}^{+}(\theta)\cup Q_{r}^{-}(\theta)
	\end{equation*} and when not located at the origin, we denote
	\begin{align*}
	(x_0,t_0)+Q_{r}^{-}(\theta)&=B_r(x_0)\times(t_0-\theta r^q,t_0],\\
	(x_0,t_0)+Q_{r}^{+}(\theta)&=B_r(x_0)\times(t_0,t_0+\theta r^q),\\
	(x_0,t_0)+Q_{r}(\theta)&=B_r(x_0)\times(t_0-\theta r^q,t_0+\theta r^q).
	\end{align*}
	Our main result is that non-negative viscosity solutions to \eqref{eq:rgnppar} satisfy the following \textit{elliptic Harnack's inequality} if the following range condition holds
	\begin{equation}
	\label{eq:range}
	2>q>\begin{cases}
	1 & \text{ if }p\geq\frac{1+n}{2},\\
	\frac{2(n-p)}{n-1}& \text{ if }1<p<\frac{1+n}{2}.
	\end{cases}
	\end{equation}
	We inspect the optimality of this range after the formulation of the theorem.
	\begin{theorem}[Elliptic Harnack's inequality]
		\label{thm:ellipticharnack}
		Let $u \geq 0$ be a viscosity solution to \eqref{eq:rgnppar} in $Q_{1}^{-}(1)$ and the range condition \eqref{eq:range} holds. Fix $\left(x_{0}, t_{0}\right) \in Q_{1}^{-}(1)$. Then for any $\sigma > 1$ there exist $\gamma=\gamma(n, p, q, \sigma)$ and ${c}={c}(n, p, q, \sigma)$ such that
		\begin{equation*}
		\gamma^{-1}\sup_{B_r(x_0)}u(\cdot,t_0)\leq u(x_0,t_0)\leq\gamma\inf_{B_{r}(x_0)}u(\cdot,t_0),
		\end{equation*}
		whenever $(x_0,t_0)+Q_{\sigma r}(\theta)\subset Q_{1}^{-}(1)$ where
		\begin{equation*}
		\theta={c}u\left(x_{0}, t_{0}\right)^{2-q}.
		\end{equation*}
	\end{theorem} 
	Our proof relies on comparison principle and parabolic Harnack's inequality proven for viscosity solutions of \eqref{eq:rgnppar} in \cite{Parviainen2020} and construction of an explicit viscosity supersolution with infinite boundary values. Existence of such solutions relies on the singularity of the equation and was proven for the $p$-parabolic case in \cite[Theorem 4.1]{Bonforte2010}. Here we constructed a concrete solution in order to obtain an explicit proof at each step. If $q$ approaches either end point of range \eqref{eq:range} the constant $\gamma$ tends to infinity and $c$ approaches zero.
	
	Elliptic Harnack's inequality may fail outside of the range condition \eqref{eq:range}. To illustrate
	this, recall a result by Parviainen and Vázquez \cite{Parviainen2020} according
	to which radial viscosity solutions to \eqref{eq:rgnppar} are
	equivalent to weak solutions of the one-dimensional equation
	\begin{equation}
	\partial_{t}u-\kappa\Delta_{q,d}u=0\quad\text{in }(-R,R)\times(0,T).\label{eq:radial eq}
	\end{equation}
	Here $\kappa:=(p-1)/(q-1)$ and (denoting by $u_r$ the radial derivative of $u$)
	\[
	\Delta_{q,d}u:=\left|u_{r}\right|^{q-2}\left((q-1)u_{rr}+\frac{d-1}{r}u_{r}\right) 
	\]
	is heuristically the usual radial $q$-Laplacian in a fictitious dimension
	\begin{equation*}
	d:=\frac{(n-1)(q-1)}{p-1}+1.
	\end{equation*} If $d$ happens to be an integer, then solutions
	to \eqref{eq:radial eq} are equivalent to radial weak solutions of
	the $q$-parabolic equation in $B_R\times(0,T)\subset\mathbb{R}^{d+1}$. On the other hand, the counterexamples in \cite[p.~140]{Dibenedetto2012} show that elliptic Harnack's inequality for the $q$-parabolic equation in $\mathbb{R}^{d+1}$ holds only in the range $2d/(d+1)<q<2$, from which one can derive the range condition \eqref{eq:range} by recalling the definition of $d$. In fact, the counterexample in \cite{Dibenedetto2012} directly translates into our context even when $d$ is not an integer. To see this, suppose that $1<p<(1+n)/2$ and set $q=2(n-p)/(n-1)$. Then in particular $q>1$. We define in radial coordinates 
	\begin{equation*}
	u(r,t):=(\left|r\right|^{\frac{2d}{d-1}}+e^{\kappa bt})^{-(d-1)/2}\quad\text{for all }r\in\mathbb{R}.
	\end{equation*}
	By a direct computation $u$ satisfies \eqref{eq:radial eq} classically in $(-R,-\delta)\cup(\delta,R)$ for any $R>0$ and small $\delta>0$. Letting $\delta\rightarrow 0$ then shows that $u$ is a weak solution in the sense of \cite{Parviainen2020} and therefore a viscosity solution to \eqref{eq:rgnppar} in $\mathbb{R}^{n+1}$. However, $u$ fails to satisfy elliptic Harnack's inequality since $u(0,t)/u(1,t)\rightarrow 0$ as $t\rightarrow-\infty$.
	
	Finally, we point out that in the case $q=p$, the range condition becomes
	\begin{equation*}
	2>p>\frac{2n}{n+1}=:p_*
	\end{equation*} the so called supercritical $p$-parabolic equation for which we have both intrinsic \cite{Dibenedetto2008, Kuusi2008} and elliptic Harnack's inequality \cite{Dibenedetto1993}. As mentioned, in the subcritical case $p\leq p_*$ both of the inequalities fail \cite{Dibenedetto2012} but there are some known Harnack type results, see for example \cite[Proposition 1.1]{Dibenedetto2009}.
	\section{Preliminaries}
	Apart from the case $q=p$, the equation \eqref{eq:rgnppar} is in non-divergence form and we cannot use integration by parts to define standard weak solutions and will thus use the concept of viscosity solutions. Moreover the equation is singular when $2>q>1$, and thus we need to restrict the class of test functions to retain good priori control on the behaviour near the singularities and make sure the limits remain well defined. We use the  definition with admissible test functions introduced in \cite{Ishii1995} for a different  class of equations and in \cite{Ohnuma1997} for our setting. This is the standard definition in this context. In the case of the $p$-parabolic equation, that is $q=p$, the notions of weak and viscosity solution are equivalent for all $p\in (1,\infty)$ \cite{Juutinen2001, Parviainen2020, Siltakoski2021}.

	Let $\Om\subset\Rn$ be a domain and denote $\Om_T=\Om\times(0,T)$ the space-time cylinder and
	\begin{equation*}
	\partial_{p}\Om=\left(\Om\times\{0\}\right)\cup\left({\partial\Om\times[0,T]}\right)
	\end{equation*}
	its parabolic boundary. Denote
	\begin{equation}
	\label{eq:gnnpparf}
	F(\eta,X)=\abs{\eta}^{q-2}\Tr\left(X-(p-2)\frac{\eta\otimes \eta}{\abs{\eta}^2}X\right)
	\end{equation}
	so that
	\begin{align*}
	F(\nabla u,D^2u)&=\abs{\nabla u}^{q-2}(\Delta u+(p-2)\Delta_\infty^Nu)=\Delta_p^qu
	\end{align*}
	whenever $\nabla u\not=0$.
	Let $\mathcal{F}(F)$ be the set of functions $f\in C^2([0,\infty))$ such that
	\begin{equation*}
	f(0)=f'(0)=f''(0)=0 \text{ and } f''(r)>0 \text{ for all }r>0,
	\end{equation*}
	and also require that for $g(x):=f(\abs{x})$, it holds that
	\begin{equation*}
	\lim_{\substack{x\to0\\x\not=0}}F(\nabla g,D^2g)=0.
	\end{equation*}
	This set $\mathcal{F}(F)$ is never empty because it is easy to see that $f(r)=r^\beta\in\mathcal{F}(F)$ for any $\beta>\max(q/(q-1),2)$. Note also that if $f\in\mathcal{F}(F)$, then $\lambda f\in\mathcal{F}(F)$ for all $\lambda>0$.
	
	Define also the set
	\begin{equation*}
	\Sigma=\{\sigma\in C^1(\R)\mid \sigma \text{ is even}, \sigma(0)=\sigma'(0)=0, \text{ and }\sigma(r)>0 \text{ for all }r>0\}.
	\end{equation*}
	We use these $\mathcal{F}(F)$ and $\Sigma$ to define admissible set of test functions for viscosity solutions.
	\begin{definition}\sloppy
		A function $\vp\in C^2(\Om_T)$ is admissible if for any $(x_0,t_0)\in\Om_T$ with ${\nabla \vp(x_0,t_0)=0}$, there are $\delta>0$, $f\in\mathcal{F}(F)$ and $\sigma\in\Sigma$ such that
		\begin{equation*}
		\abs{\vp(x,t)-\vp(x_0,t_0)-\partial_{t}\vp(x_0,t_0)(t-t_0)}\leq f(\abs{x-x_0})+\sigma(t-t_0),
		\end{equation*}
		for all $(x,t)\in B_\delta(x_0)\times (t_0-\delta, t_0+\delta)$.
	\end{definition}
	Note that by definition a function $\vp$ is automatically admissible in $\Om_T$ if either $\nabla\vp(x,t)\not=0$ in $\Om_T$ or the function $-\vp$ is admissible in $\Om_T$.
	\begin{definition}
		A function $u:\Om_T\rightarrow\mathbb{R}\cup\left\{ \infty\right\} $
		is a viscosity supersolution to
		\[
		\partial_{t}u=\Delta_{p}^{q}u\quad\text{in }\Om_T
		\]
		if the following three conditions hold.
		\begin{enumerate}
			\label{def:super}
			\item $u$ is lower semicontinuous,
			\item $u$ is finite in a dense subset of $\Om_T$,
			\item whenever an admissible $\vp\in C^{2}(\Om_T)$ touches $u$ at $(x,t)\in\Om_T$
			from below, we have
			\[
			\begin{cases}
			\partial_{t}\vp(x,t)-\Delta_{p}^{q}\vp(x,t)\geq0 & \text{if }\nabla \vp(x,t)\not=0,\\
			\partial_{t}\vp(x,t)\geq0 & \text{if }\nabla \vp(x,t)=0.
			\end{cases}
			\]
		\end{enumerate}
		A function $u:\Om_T\rightarrow\mathbb{R}\cup\left\{ -\infty\right\} $
		is a viscosity subsolution if $-u$ is a viscosity supersolution. A function $u:\Om_T\rightarrow\mathbb{R} $
		is a viscosity solution if it is a supersolution and a subsolution.
	\end{definition}
	
	Our proof uses the following comparison principle, which is Theorem 3.1 in \cite{Ohnuma1997}.
	\begin{theorem}
		\label{thm:comp}
		Let $\Omega\subset\Rn$ be a bounded domain. Suppose that u is viscosity supersolution and $v$ is a viscosity subsolution to \eqref{eq:rgnppar} in $\Omega_{T}$. If
		$$
		\infty \neq \limsup _{\Omega_{T} \ni(y, s) \rightarrow(x, t)} v(y, s) \leq \liminf _{\Omega_{T} \ni(y, s) \rightarrow(x, t)} u(y, s) \neq-\infty
		$$
		for all $(x, t) \in \partial_{p} \Omega_{T},$ then $v \leq u$ in $\Omega_{T}$.
	\end{theorem}
	We also use the following forward Harnack's inequality, which is Theorem 7.3 in \cite{Parviainen2020}.
	\begin{theorem}
		\label{thm:parharnack}
		Let $u \geq 0$ be a viscosity solution to \eqref{eq:rgnppar} in $Q_{1}^{-}(1)$ and the range condition \eqref{eq:range} holds or $q\geq2$. Fix $\left(x_{0}, t_{0}\right) \in Q_{1}^{-}(1)$ such that $u(x_0,t_0)>0$. Then there exist ${\mu=\mu(n, p, q)}$ and ${c}={c}(n, p, q)$ such that
		\begin{equation*}
		u\left(x_{0}, t_{0}\right) \leq \mu \inf _{B_{r}\left(x_{0}\right)} u\left(\cdot, t_{0}+\theta r^q\right)
		\end{equation*}
		where
		\begin{equation*}
		\theta={c}u\left(x_{0}, t_{0}\right)^{2-q},
		\end{equation*}
		whenever $(x_0,t_0)+Q_{4r}(\theta) \subset Q_{1}^{-}(1)$.
	\end{theorem}
	\begin{remark}
		\label{re:positivity}
		Note that the assumption $u(x_0,t_0)>0$ is needed only in the case $q\geq2$. Assuming $q$ satisfies the range condition \eqref{eq:range}, we can define $v(x,t)=u(x,t)+\eps>0$ for some small constant $\eps>0$. Using Theorem \ref{thm:parharnack} for this $v$, we get
		\begin{equation*}
		u\left(x_{0}, t_{0}\right)+\eps \leq \mu \inf _{B_{r}\left(x_{0}\right)} u\left(\cdot, t_{0}+c\left(u\left(x_{0}, t_{0}\right)+\eps\right)^{2-q} r^q\right)+\eps
		\end{equation*}
		and letting $\eps\to0$ gives us the intrinsic Harnack's inequality for $u$ by continuity.		
	\end{remark}
	\section{A viscosity supersolution with infinite boundary values}
	In this section we construct an explicit viscosity supersolution
	$v$ to \eqref{eq:rgnppar} in $B_{R}(0)\times(0,\infty)$ that
	takes infinite lateral boundary values and vanishes at the bottom
	of the cylinder. Recently infinite point source solutions have been constructed for supercritical $p$-parabolic equation in \cite{Giri2021}. While it is straightforward to the check that our function is a supersolution, it may not be	
	immediately clear how one obtains its expression and therefore we present the derivation. The construction is based on the equivalence result between radial viscosity solutions of \eqref{eq:rgnppar} and weak solutions of \eqref{eq:radial eq}, see \cite[Theorem 4.2]{Parviainen2020}.
	Solutions to the one-dimensional equation \eqref{eq:radial eq} can be at least formally
	obtained via the stationary equation
	\begin{equation}
	-\kappa\Delta_{q,d}v+\frac{v}{2-q}=0.\label{eq:stationary}
	\end{equation}
	Indeed, if $v$ solves (\ref{eq:stationary}) and we set $u(r,t)=t^{\frac{1}{2-q}}v(r)$,
	then we have formally
	\begin{equation*}
	\kappa\Delta_{q,d}u=\kappa\left|u_{r}\right|^{q-2}\left((q-1)u_{rr}+\frac{d-1}{r}u_{r}\right)=\kappa t^{\frac{1}{2-q}-1}\Delta_{q,d}v=\frac{1}{2-q}t^{\frac{1}{2-q}-1}v=\partial_{t}u,
	\end{equation*}
	so $u$ solves (\ref{eq:radial eq}). Now, the equation (\ref{eq:stationary})
	can be seen as a radial version of the equation
	\begin{equation}
	-\kappa\Delta_{q}v+\frac{v}{2-q}=0\label{eq:stationary q}
	\end{equation}
	in a fictitious dimension $d$. Here $\Delta_q$ denotes the usual $q$-Laplacian. Equations such as (\ref{eq:stationary q})
	have been widely studied in the literature when $d$ is an integer.
	In particular, D{\'{i}}az and Letelier \cite{Diaz1993} obtained the
	existence of local solutions with infinite boundary values to a large
	class of equations that includes (\ref{eq:stationary q}). In their
	proof they make use of an explicit radial supersolution with infinite
	boundary values (see \cite[Theorem 5.1]{Diaz1993}). Our idea
	is to take this supersolution and use the above transformations to
	obtain a supersolution to \eqref{eq:rgnppar}. This way one arrives
	to the expression \eqref{eq:explicit formula} below. 
	\begin{lemma}
		\label{lem:supersolution}
		Suppose that $1<q<2$, $p>1$ and let $R>0$. Then there exists a
		positive constant $\lambda=\lambda(n,p,q)$ such that the function
		\begin{equation}
		v(x,t):=\lambda t^{\frac{1}{2-q}}\left(\frac{1}{R^{\frac{1}{1-q}}(R^{\frac{q}{q-1}}-\left|x\right|^{\frac{q}{q-1}})}\right)^{\frac{q}{2-q}}\label{eq:explicit formula}
		\end{equation}
		is a viscosity supersolution to \eqref{eq:rgnppar} in
		$B_{R}(0)\times(0,\infty)$.
	\end{lemma}
	\begin{proof}
		Let us first consider the case $R=1.$
		
		\textbf{(Step 1)} For $(r,t)\in[0,1)\times(0,\infty)$, we set
		\[
		w(r,t):=\lambda t^{\frac{1}{2-q}}(1-r^{\frac{q}{q-1}})^{\frac{q}{q-2}},
		\]
		where $\lambda=\lambda(n,p,q)$ is a large constant to be chosen later.
		We show that $w$ satisfies
		\begin{equation}
		\partial_{t}w-\left|w^{\prime}\right|^{q-2}\left((p-1)w^{\prime\prime}+\frac{n-1}{r}w^{\prime}\right)\geq0\quad\text{in }(0,1)\times(0,\infty).\label{eq:explicit super 1}
		\end{equation}
		We have
		\begin{align*}
		\label{eq:explicit super 2}
		\partial_{t}w(r,t) & =\lambda\frac{1}{2-q}t^{\frac{1}{2-q}-1}(1-r^{\frac{q}{q-1}})^{\frac{q}{q-2}},\\
		w^{\prime}(r,t) & =-\lambda\frac{q^{2}}{(q-1)(q-2)}\cdot t^{\frac{1}{2-q}}r^{\frac{q}{q-1}-1}(1-r^{\frac{q}{q-1}})^{\frac{q}{q-2}-1}\numberthis\\
		\end{align*}
		and
		\begin{equation}
		\begin{aligned}
		\label{eq:explicit super 3}
		w^{\prime\prime}(r,t) & =-\lambda\frac{q^{2}}{(q-1)(q-2)}\left(\frac{q}{q-1}-1\right)\cdot t^{\frac{1}{2-q}}r^{\frac{q}{q-1}-2}(1-r^{\frac{q}{q-1}})^{\frac{q}{q-2}-1}\\
		& \ \ \ +\lambda\frac{q^{3}}{(q-1)^{2}(q-2)}\left(\frac{q}{q-2}-1\right)\cdot t^{\frac{1}{2-q}}r^{2(\frac{q}{q-1}-1)}(1-r^{\frac{q}{q-1}})^{\frac{q}{q-2}-2}.
		\end{aligned}
		\end{equation}
		Thus by combining \eqref{eq:explicit super 2} and \eqref{eq:explicit super 3} we get
		\begin{align*}
		& (p-1)w^{\prime\prime}(r,t)+\frac{n-1}{r}w^{\prime}(r,t)\\
		& \ =-\lambda\frac{q^{2}(p-1)}{(q-1)^{2}(q-2)}t^{\frac{1}{2-q}}r^{\frac{q}{q-1}-2}(1-r^{\frac{q}{q-1}})^{\frac{q}{q-2}-1}\\
		& \ \ \ \
		+\lambda\frac{2q^{3}(p-1)}{(q-1)^{2}(q-2)^{2}}t^{\frac{1}{2-q}}r^{2(\frac{q}{q-1}-1)}(1-r^{\frac{q}{q-1}})^{\frac{q}{q-2}-2}\\
		& \ \ \ \ -\lambda\frac{q^{2}(n-1)}{(q-1)(q-2)}t^{\frac{1}{2-q}}r^{\frac{q}{q-1}-2}(1-r^{\frac{q}{q-1}})^{\frac{q}{q-2}-1}\\
		& \leq C(n,p,q)\lambda t^{\frac{1}{2-q}}r^{\frac{q}{q-1}-2}(1-r^{\frac{q}{q-1}})^{\frac{q}{q-2}-2}((1-r^{\frac{q}{q-1}})+r^{\frac{q}{q-1}}).\\
		& \leq C(n,p,q)\lambda t^{\frac{1}{2-q}}r^{\frac{q}{q-1}-2}(1-r^{\frac{q}{q-1}})^{\frac{q}{q-2}-2}.
		\end{align*}
		Combining this with the formula \eqref{eq:explicit super 2}, we obtain
		\begin{align*}
		& \left|w^{\prime}\right|^{q-2}((p-1)w^{\prime\prime}+\frac{n-1}{r}w^{\prime})\\
		& \ \leq C(n,p,q)\lambda^{q-2}t^{\frac{q-2}{2-q}}r^{(q-2)(\frac{q}{q-1}-1)}(1-r^{\frac{q}{q-1}})^{(q-2)(\frac{q}{q-2}-1)}\\
		& \ \ \ \ \ \ \ \ \ \ \ \ \ \ \ \cdot\lambda t^{\frac{1}{2-q}}r^{\frac{q}{q-1}-2}(1-r^{\frac{q}{q-1}})^{\frac{q}{q-2}-2}\\
		& \ =C(n,p,q)\lambda^{q-1}t^{\frac{1}{2-q}-1}(1-r^{\frac{q}{q-1}})^{\frac{q}{q-2}}
		\end{align*}
		where we used that $(q-2)(\frac{q}{q-1}-1)+(\frac{q}{q-1}-2)=\frac{q-2}{q-1}+\frac{2-q}{q-1}=0$
		and $(q-2)(\frac{q}{q-2}-1)=2$. Hence,
		\begin{align*}
		& \partial_{t}w-\left|w^{\prime}\right|^{q-2}((p-1)w^{\prime\prime}+\frac{n-1}{r}w^{\prime})\\
		& \ \geq\lambda\frac{1}{2-q}t^{\frac{1}{2-q}-1}(1-r^{\frac{q}{q-1}})^{\frac{q}{q-2}}-C(n,p,q)\lambda^{q-1}t^{\frac{1}{2-q}-1}(1-r^{\frac{q}{q-1}})^{\frac{q}{q-2}}\\
		& \ =\lambda^{q-1}t^{\frac{1}{2-q}-1}(1-r^{\frac{q}{q-1}})^{\frac{q}{q-2}}\left(\frac{\lambda^{2-q}}{2-q}-C(n,p,q)\right).
		\end{align*}
		By taking $\lambda=\lambda(n,p,q)$ large enough, the right-hand side
		of the above display can be made non-negative. This way we see that
		$w$ we satisfies (\ref{eq:explicit super 1}).
		
		\textbf{(Step 2)} We set
		\[
		v(x,t):=w(\left|x\right|,t)\quad\text{for all }(x,t)\in B_{1}\times(0,\infty).
		\]
		Suppose first that $(x,t)\in(B_{1}\setminus\left\{ 0\right\} )\times(0,\infty)$
		and denote $r=\left|x\right|$. Then we have 
		\begin{align*}
		\nabla v(x,t) & =\frac{x}{r}w^{\prime}(r,t),\\
		D^{2}v(x,t) & =\frac{x}{r}\otimes\frac{x}{r}w^{\prime\prime}(r,t)+\frac{1}{r}(I-\frac{x}{r}\otimes\frac{x}{r})w^{\prime}(r,t).
		\end{align*}
		Therefore, since $w$ satisfies (\ref{eq:explicit super 1}), we have
		\begin{align*}
		\partial_{t}v-\Delta_{p}^{q}v & =\partial_{t}v-\left|\nabla v\right|^{q-2}\Tr\left(D^2 v+(p-2)\frac{\nabla v \otimes \nabla v}{\left|\nabla v\right|^{2}}D^2v \right)\\
		& =\partial_{t}w-\left|w^{\prime}\right|^{q-2}((p-1)w^{\prime\prime}+\frac{(n
			-1)}{r}w^{\prime})\geq0.
		\end{align*}
		This means that $v$ is a classical supersolution in $(B_{1}\setminus\left\{ 0\right\} )\times(0,\infty)$.
		We still need to consider the set $\left\{ 0\right\} \times(0,\infty)$.
		Since $1<q<2$, it follows from the formulas \eqref{eq:explicit super 2} of $w^{\prime}$ and \eqref{eq:explicit super 3} of
		$w^{\prime\prime}$ that $v\in C^{2}(B_{1}\times(0,\infty))$ with
		$\nabla v(0,t)=0$ and $\partial_{t}v(0,t)\geq0$ for all $t>0$. Therefore,
		if $\vp\in C^{2}$ touches $v$ from below at $(0,t)$, we have
		$\nabla \vp(0,t)=\nabla v(0,t)=0$ and $\partial_{t}\vp(0,t)=\partial_{t}v(0,t)\geq0$,
		as required. Consequently $v$ is a viscosity supersolution in $B_{1}\times(0,\infty)$.
		
		\textbf{(Step 3)} It remains to consider $R>0$. Let $v$ be the viscosity
		supersolution to
		\[
		\partial_{t}v=\Delta_{p}^{q}v\quad\text{in }B_{1}\times(0,\infty)
		\]
		which we constructed in the previous steps. Set $\widetilde{v}(x,t):=v(R^{-1}x,R^{-q}t)$.
		Then for all $(x,t)\in B_{R}(0)\setminus\left\{ 0\right\} \times(0,\infty)$
		we have
		\[
		\partial_{t}\widetilde{v}(x,t)-\Delta_{p}^{q}\widetilde{v}(x,t)=R^{-q}v(R^{-1}x,R^{-q}t)-R^{-q}\Delta_{p}^{q}v(R^{-1}x,R^{-q}t)\geq0,
		\]
		so $\widetilde{v}$ is a viscosity supersolution in $B_{R}(0)\times(0,\infty)$.
		Moreover,
		\begin{align*}
		\widetilde{v}(x,t)=\lambda(R^{-q}t)^{\frac{1}{2-q}}\left(1-\left|R^{-1}x\right|^{\frac{q}{q-1}}\right)^{\frac{q}{q-2}} & =\lambda t^{\frac{1}{2-q}}(RR^{\frac{q}{1-q}}(R^{\frac{q}{q-1}}-\left|x\right|^{\frac{q}{q-1}}))^{\frac{q}{q-2}}\\
		& =\lambda t^{\frac{1}{2-q}}(R^{\frac{1}{1-q}}(R^{\frac{q}{q-1}}-\left|x\right|^{\frac{q}{q-1}}))^{\frac{q}{q-2}}
		\end{align*}
		as desired.
	\end{proof}
	\section{A parabolic Harnack's inequality}
	In this section we prove a both-sided version of parabolic Harnack's inequality for equation \eqref{eq:rgnppar} which is of independent interest and needed for our proof of Theorem \ref{thm:ellipticharnack}. The proof of the backwards estimate is an adaptation of Section 6.9.\ in \cite{Dibenedetto2012} apart from the non-emptyness of the set $\mathcal{U}_\alpha$ below, which we prove using the comparison principle and the explicit supersolution we constructed in Lemma \ref{lem:supersolution}. To this end, we need to reduce the waiting time in the forward Harnack inequality. This kind of reduction can be achieved by increasing the multiplier $\mu$, as made precise in the following proposition.
	
	\begin{proposition}
		\label{prop:waitingtime}
		Let $u\geq0$ be a viscosity solution to \eqref{eq:rgnppar} in $Q_{1}^{-}(1)$ and the range
		condition \eqref{eq:range} holds. Fix $(x_{0},t_{0})\in Q_{1}^{-}(1)$ such that $u(x_{0},t_{0})>0$.
		Let $c$ be as in Theorem \ref{thm:parharnack}. Then for any $\hat{c}\in(0,c)$
		there exists $\hat{\mu}=\hat{\mu}(n,p,q,\hat{c})$ such that
		\begin{equation}
		u(x_{0},t_{0})\leq\hat{\mu}\inf_{B_{r}(x_{0})}u(\cdot,t_{0}+\hat{\theta}r^{q}),\label{eq:claim}
		\end{equation}
		whenever $(x_{0},t_{0})+Q_{5r}(\hat{\theta})\subset Q_{1}^{-}(1)$,
		where $\hat{\theta}=\hat{c}u(x_{0},t_{0})^{2-q}$.
	\end{proposition}
	
	We postpone the proof of Proposition \ref{prop:waitingtime} to the end of this section and consider the both-sided Harnack inequality next.
	
	\begin{theorem}
		\label{thm:backharnack}
		Let $u \geq 0$ be a viscosity solution to \eqref{eq:rgnppar} in $Q_{1}^{-}(1)$ and the range condition \eqref{eq:range} holds. Fix $\left(x_{0}, t_{0}\right) \in Q_{1}^{-}(1)$. Then there exist $\mu=\mu(n, p, q)$, ${c}={c}(n, p, q)$ and ${\alpha=\alpha(n,p,q)\in(0,1)}$ such that
		\begin{equation*}
		\mu^{-1}\sup_{B_r(x_0)}u(\cdot,t_0-\theta r^{q})\leq u(x_0,t_0)\leq\mu\inf_{B_{r}(x_0)}u(\cdot,t_0+\theta r^{q})
		\end{equation*}
		where
		\begin{equation*}
		\theta={c}u\left(x_{0}, t_{0}\right)^{2-q},
		\end{equation*}
		whenever $(x_0,t_0)+Q_{\frac{6}{\alpha}r}(\theta) \subset Q_{1}^{-}(1)$.
	\end{theorem}
	\begin{proof}
		Without loss of generality, we may assume $u(x_0,t_0)>0$ as stated in Remark \ref{re:positivity}. Let $\hat{c}$ be a small positive constant to be chosen later.
		 For this $\hat{c}$, let $\hat{\mu} > 2$ be given by Proposition \ref{prop:waitingtime}. Let $\rho$ be a radius such that $(x_0,t_0)+Q_{6\rho}(\hat{\theta}) \subset Q_{1}^{-}(1)$, $\hat{\theta} = \hat{c}u(x_0, t_0)^{2-q}$, and let
		\begin{equation}
		\overline{t}=t_0-{\hat{c}}u(x_0,t_0)^{2-q}\rho^q.
		\end{equation}
		Let $\alpha\in(0,1)$ be a constant to be chosen later and define the sets
		\begin{equation}
		\mathcal{U}_\alpha=B_{\alpha \rho}(x_0)\cap \{x\in \overline{B}_{\alpha \rho}(x_0)\mid u(x,\overline{t})\leq \hat{\mu} u(x_0,t_0)\}=:B_{\alpha \rho}(x_0)\cap D.
		\end{equation}
		We will first show that $\alpha$ can be chosen to make $\mathcal{U}_\alpha$ open.
		Assume that $\mathcal{U}_\alpha$ is not empty and fix $z\in\mathcal{U}_\alpha$. Since $u$ is continuous, we can choose a radius $\eps$ such that $B_\eps(z)\subset B_{\alpha \rho}(x_0)$ and
		\begin{equation}
		\label{eq:backharnack1}
		u(y,\overline{t})\leq2 \hat{\mu} u(x_0,t_0) \quad \text{ for all }y\in B_\eps(z).
		\end{equation}
		For each $y\in B_\eps(z)$, construct the intrinsic $q$-paraboloid
		\begin{equation*}
		\mathcal{P}(y,\overline{t})=\{(x,t)\in Q_{1}^{-}(1)\mid t-\overline{t}\geq \hat{c}u(y,\overline{t})^{2-q}\abs{x-y}^q\}.
		\end{equation*}

		Selecting
		\begin{equation}
		\label{eq:alpha_choice}
		\alpha:=\left(2\hat{\mu}\right)^{\frac{q-2}{q}},
		\end{equation}
		we have  $(x_0, t_0) \in \mathcal{P}(y, \overline{t})$ whenever  $y \in B_\varepsilon (z)$, since using \eqref{eq:backharnack1} we can estimate

		\begin{align*}
		\hat{c}u(y,\overline{t})^{2-q}\left|y-x_{0}\right|^{q}\leq \hat{c}(2\hat{\mu})^{2-q}u(x_{0},t_{0})^{2-q}\left|y-x_{0}\right|^{q} & \leq \hat{c}(2\hat{\mu})^{2-q}u(x_{0},t_{0})^{2-q}(\alpha\rho)^{q}\\
		 & \leq \hat{c}u(x_{0},t_{0})^{2-q}\rho^{q}=t_{0}-\overline{t}.
		\end{align*}

		Assume for a moment that $u(y,\overline{t}) \geq 2u(x_0, t_0)$ and pick a radius 
		\begin{equation*}
		\hat{\rho}=\frac{u(x_0,t_0)^{\frac{2-q}{q}}}{u(y,\overline{t})^{\frac{2-q}{q}}}\rho
		\end{equation*} so that
		\begin{equation*}
		\overline{t}+\hat{c}u\left(y, \overline{t}\right)^{2-q}\hat{\rho}^{q}=\overline{t}+\hat{c}u\left(x_0, t_0\right)^{2-q}\rho^{q}=t_0.
		\end{equation*}
		Thus by Proposition \ref{prop:waitingtime} we have
		\begin{equation}
		\label{eq:backharnack2}
		u\left(y, \overline{t}\right) \leq \hat{\mu} \inf_{B_{\hat{\rho}}(y)} u\left(\cdot, \overline{t}+\hat{c}u\left(y, \overline{t}\right)^{2-q}\hat{\rho}^{q}\right) = \hat{\mu} \inf_{B_{\hat{\rho}}(y)} u\left(\cdot, t_0\right)\leq \hat{\mu} u(x_0,t_0),
		\end{equation}
		where the last inequality holds because from $(x_0,t_0)\in\mathcal{P}(y,\overline{t})$, it follows
		\begin{equation*}
		\abs{x_0-y}^q\le\frac{t_0-\overline{t}}{\hat{c}u(y,\overline{t})^{2-q}}=\frac{\hat{c}u(x_0,t_0)^{2-q}\rho^q}{\hat{c}u(y,\overline{t})^{2-q}}=\hat{\rho}^q.
		\end{equation*}
		The use of Proposition \ref{prop:waitingtime} here is justified since $B_{5\hat{\rho}}(y)\subset B_{6\rho}(x_0)$ because
		\begin{equation*}
		5\hat{\rho}+\rho= 5\frac{u(x_0,t_0)^{\frac{2-q}{q}}}{u(y,\overline{t})^{\frac{2-q}{q}}}\rho+\rho\leq5\left(\frac{u(x_0,t_0)}{2u(x_0, t_0)}\right)^{\frac{2-q}{q}}\rho+\rho\leq6\rho,
		\end{equation*}
		where we use our assumption $u(y,\overline{t}) \geq 2u(x_0, t_0)$ and $q<2$. In the time direction it holds
		\begin{align*}
			\overline{t}-\hat{c}u(y,\overline{t})^{2-q}(5\hat{\rho})^q&=\overline{t}-\hat{c}u(y,\overline{t})^{2-q}\left(\frac{u(x_0,t_0)}{u(y,\overline{t})}\right)^{2-q}(5\rho)^q\\
			&=t_0-\hat{c}u(x_0,t_0)^{2-q}\rho^q-\hat{c}u(x_0,t_0)^{2-q}(5\rho)^q\\
			&=t_0-(1+5^q)\hat{\theta}\rho^q>t_0-\hat{\theta}(6\rho)^q
		\end{align*}
		and thus there is enough room to use the proposition. The last inequality holds because $q>1$.
		
		If $u(y,\overline{t})<2u(x_0,t _0)$, then \eqref{eq:backharnack2} holds automatically since $\hat{\mu} \geq 2$.
		We can get inequality \eqref{eq:backharnack2} for any $y\in B_{\eps}(z)$ and thus $B_{\eps}(z)\subset\mathcal{U}_\alpha$ for a radius $\eps$ only depending on $z$. This can be repeated for any $z\in\mathcal{U}_\alpha$ and thus the set $\mathcal{U}_\alpha$ has to be open.
		\par
		We still need to show that $\mathcal{U}_\alpha\not=\emptyset$. If we assume thriving for a contradiction that $\mathcal{U}_\alpha=\emptyset$, then
		\begin{equation}
			\label{eq:backharnack3}
			m:=\inf_{B_{\alpha\rho}(x_0)}u(\cdot,\overline{t})\geq\hat{\mu} u(x_0,t_0).
		\end{equation}
		Consider the function
		\begin{equation*}
			w(x,t):=-\lambda(t-\overline{t})^{\frac{1}{2-q}}\left(\frac{1}{(\alpha\rho)^{\frac{1}{1-q}}\left((\alpha\rho)^{\frac{q}{q-1}}-\abs{x-x_0}^{\frac{q}{q-1}}\right)}\right)^{\frac{q}{2-q}}+m.
		\end{equation*}
		By Lemma \ref{lem:supersolution}, $w$ is a viscosity subsolution to \eqref{eq:rgnppar} in $B_{\alpha\rho}(x_0)\times(\overline{t},\infty)$ and satisfies
		\begin{equation*}
			\begin{cases}
				w(x,\overline{t})\equiv m \leq u(x,\overline{t}) & \text{for all }x\in B_{\alpha\rho}(x_0),\\
				\lim\limits_{\Omega_{T} \ni(x,t)\to(y,s)}w(x,t)=-\infty & \text{for all }(y,s)\in\partial B_{\alpha\rho}(x_0)\times(\overline{t},\infty).\\
			\end{cases}
		\end{equation*}
	Thus by comparison principle Theorem \ref{thm:comp}, we have $u\geq w$ in $B_{\alpha\rho}(x_0)\times[\overline{t},\infty)$, so in particular we have
	\begin{align*}
		u(x_0,t_0)&\geq w(x_0,t_0)\\
		&=-\lambda\left(t_0-t_0+\hat{c}u(x_0,t_0)^{2-q}\rho^q\right)^{\frac{1}{2-q}}\left(\frac{1}{(\alpha\rho)^{\frac{1}{1-q}}\left((\alpha\rho)^{\frac{q}{q-1}}-0\right)}\right)^{\frac{q}{2-q}}+m \\
		&=-\lambda \hat{c}^{\frac{1}{2-q}}\rho^{\frac{q}{2-q}}(\alpha\rho)^{-\frac{q}{2-q}}u(x_0,t_0)+m\\
		&\geq-\lambda \hat{c}^{\frac{1}{2-q}}\alpha^{-\frac{q}{2-q}}u(x_0,t_0)+\hat{\mu} u(x_0,t_0)\\
		&=\left(-2\lambda \hat{c}^{\frac{1}{2-q}}+1\right)\hat{\mu} u(x_0,t_0)\\
		&> 2\left(-2\lambda \hat{c}^{\frac{1}{2-q}}+1\right) u(x_0,t_0),
	\end{align*}

	where the last two inequalities follow from our assumption \eqref{eq:backharnack3} and that $\hat{\mu} > 2$. By taking  $\hat{c}$ to be a small enough constant depending only on $p$, $q$ and $n$, we can ensure that the coefficient of $u(x_0,t_0)$ at the right-hand side is larger than $1$, yielding a contradiction. Thus the set $\mathcal{U}_\alpha$ cannot be empty. 
	
	We have shown that the set $\mathcal{U}_\alpha=B_{\alpha \rho}(x_0)\cap D$ is open and non-empty. Because $u$ is continuous, the set $D$ is closed and thus for our $\alpha$, we must have $B_{\alpha \rho}(x_0) \subset D$ and thus by definition of the set
		\begin{equation*}
		\sup_{B_{\alpha \rho}(x_0)}u(\cdot,\overline{t})\leq\hat{\mu} u(x_0,t_0).
		\end{equation*}
	Combining this with the right side of the Harnack's inequality Proposition \ref{prop:waitingtime}, we obtain
		\begin{equation*}
		\hat{\mu}^{-1}\sup_{B_{\alpha \rho}(x_0)}u(\cdot,t_0-\hat{c}u(x_0,t_0)^{2-q}\rho^q)\leq u(x_0,t_0)\leq\hat{\mu}\inf_{B_{\rho}(x_0)}u(\cdot,t_0+\hat{c}u(x_0,t_0)^{2-q}\rho^q)
		\end{equation*}
	for the specific $\alpha$ chosen in \eqref{eq:alpha_choice}. If we let $\tilde{{c}}=\alpha^{-q}\hat{c}$ and $r=\alpha \rho$, we have
		\begin{align*}
		\hat{\mu}^{-1}\sup_{B_{r}(x_0)}u(\cdot,t_0-\tilde{{c}} u(x_0,t_0)^{2-q}r^q)\leq u(x_0,t_0)&\leq\hat{\mu}\inf_{B_{\rho}(x_0)}u(\cdot,t_0+\tilde{{c}} u(x_0,t_0)^{2-q}r^q)
		\\&\leq\hat{\mu}\inf_{B_{r}(x_0)}u(\cdot,t_0+\tilde{{c}} u(x_0,t_0)^{2-q}r^q)
		\end{align*}
		which is what we wanted. The condition $(x_0,t_0)+Q_{6\rho}(\hat{\theta}) \subset Q_{1}^{-}(1)$ becomes the stated $(x_0,t_0)+Q_{\frac{6}{\alpha}r}(\theta) \subset Q_{1}^{-}(1)$.
	\end{proof}

	We conclude this section with the proof of Proposition \ref{prop:waitingtime}. 
	
	\begin{proof}[Proof of Proposition \ref{prop:waitingtime}]
		As discussed in Remark \ref{re:positivity}, we may assume that $u(x_{0},t_{0})>0$.
		Let $\mu>1$ and $c$ be the constants in Theorem \ref{thm:parharnack} and let $\hat{c}<c$.
		We prove (\ref{eq:claim}) for $\hat{\mu}:=\mu\tilde{\mu}$, where $\tilde{\mu}:=(c/\hat{c})^{\frac{1}{2-q}}$. 
		Denote $\hat{t}:=t_{0}+\hat{c}u(x_{0},t_{0})^{2-q}r^{q}$ and let
		$\hat{x}\in B_{r}(x_{0})$ be an arbitrary point. It now suffices
		to prove that
		\begin{equation}
		u(x_{0},t_{0})\leq\tilde{\mu}\mu u(\hat{x},\hat{t}).\label{eq:claim1}
		\end{equation}
		To this end, we may suppose that $u(x_{0},t_{0})>\tilde{\mu}u(\hat{x},\hat{t})$
		because otherwise
		\[
		u(\hat{x},\hat{t})\geq\frac{1}{\tilde{\mu}}u(x_{0},t_{0})>\frac{1}{\tilde{\mu}\mu}u(x_{0},t_{0}),
		\]
		which would already imply (\ref{eq:claim1}). Let $[(x_{0},t_{0}),(\hat{x},\hat{t})]$
		be a segment from $(x_{0},t_{0})$ to $(\hat{x},\hat{t})$, i.e.
		\begin{align*}
		[(x_{0},t_{0}),(\hat{x},\hat{t})]:=\left\{ \left(x_{0}+l\frac{\hat{x}-x_{0}}{\left|\hat{x}-x_{0}\right|},t_{0}+l\kappa\right)\mid l\in[0,\left|\hat{x}-x_{0}\right|]\right\} ,\quad & \kappa:=\frac{\hat{t}-t_{0}}{\left|\hat{x}-x_{0}\right|}.
		\end{align*}
		We have
		\[
		u(\hat{x},\hat{t})<\frac{1}{\tilde{\mu}}u(x_{0},t_{0})<u(x_{0},t_{0}).
		\]
		Thus by continuity there exists $(x_{1},t_{1})\in[(x_{0},t_{0}),(\hat{x},\hat{t})]\setminus\left\{ (x_{0},t_{0}),(\hat{x},\hat{t})\right\} $
		such that
		\begin{equation}
		u(x_{1},t_{1})=\frac{1}{\tilde{\mu}}u(x_{0},t_{0}).\label{eq:somethings 1}
		\end{equation}
		Moreover, since $(x_{1},t_{1})$ lies on the segment, there is $l_{1}\in(0,\left|\hat{x}-x_{0}\right|)$
		such that
		\[
		(x_{1},t_{1})=\left(x_{0}+l_{1}\frac{\hat{x}-x_{0}}{\left|\hat{x}-x_{0}\right|},t_{0}+l_{1}\kappa\right).
		\]
		We now have
		\begin{align}
		\left|x_{1}-\hat{x}\right| & =\left|x_{0}+l_{1}\frac{\hat{x}-x_{0}}{\left|\hat{x}-x_{0}\right|}-x_{0}-\left|\hat{x}-x_{0}\right|\frac{\hat{x}-x_{0}}{\left|\hat{x}-x_{0}\right|}\right|=(\left|\hat{x}-x_{0}\right|-l_{1})\nonumber \\
		& =\left(\frac{\hat{t}-t_{0}}{\kappa}-\frac{t_{1}-t_{0}}{\kappa}\right)=\frac{\hat{t}-t_{1}}{\kappa}.\label{eq:est1}
		\end{align}
		We set
		\[
		\rho:=\left(\frac{\hat{t}-t_{1}}{cu(x_{1},t_{1})^{2-q}}\right)^{\frac{1}{q}}
		\]
		because then, since $\kappa=(\hat{t}-t_{0})/\left|\hat{x}-x_{0}\right|$,
		we obtain using (\ref{eq:somethings 1})
		\begin{align}
		\frac{\hat{t}-t_{1}}{\kappa} & =\rho\frac{(\hat{t}-t_{1})^{1-\frac{1}{q}}(cu(x_{1},t_{1})^{2-q})^{\frac{1}{q}}}{\kappa}\nonumber \\
		& =\rho\left|\hat{x}-x_{0}\right|\frac{(\hat{t}-t_{1})^{1-\frac{1}{q}}(cu(x_{1},t_{1})^{2-q})^{\frac{1}{q}}}{(\hat{t}-t_{0})}\nonumber \\
		& <\rho r\left(\frac{cu(x_{1},t_{1})^{2-q}}{\hat{t}-t_{0}}\right)^{\frac{1}{q}}\nonumber \\
		& =\rho r\left(\frac{cu(x_{1},t_{1})^{2-q}}{\hat{c}u(x_{0},t_{0})^{2-q}r^{q}}\right)^{^{\frac{1}{q}}}\nonumber \\
		& =\rho\left(\frac{c}{\hat{c}\tilde{\mu}^{2-q}}\right)^{\frac{1}{q}}=\rho.\label{eq:est2}
		\end{align}
		Combining (\ref{eq:est1}) and (\ref{eq:est2}) we see that $\hat{x}\in B_{\rho}(x_{1})$.
		Moreover, by definition of $\rho$ we have $t_{1}+cu(x_{1},t_{1})^{2-q}\rho^{q}=\hat{t}$.
		Consequently, assuming for the moment that we have enough space to
		apply Theorem \ref{thm:parharnack} at $(x_{1},t_{1})$ for radius $\rho$, we obtain
		\[
		u(x_{1},t_{1})\leq\mu\inf_{B_{\rho}(x_{1})}u(\cdot,t_{1}+cu(x_{1},t_{1})^{2-q}\rho^{q})\leq\mu u(\hat{x},\hat{t}).
		\]
		Hence by (\ref{eq:somethings 1}) 
		\[
		u(\hat{x},\hat{t})\geq\frac{1}{\mu}u(x_{1},t_{1})=\frac{1}{\mu\tilde{\mu}}u(x_{0},t_{0}),
		\]
		as desired. 
		
		Since we use Theorem \ref{thm:parharnack} at $(x_{1},t_{1})$, $t_{1}>t_{0}$, we only
		need to check that the upper boundary of the cylinder $(x_{1},t_{1})+Q_{4\rho}(\theta)$
		is within the domain of the solution. First, by (\ref{eq:somethings 1})
		we have
		\begin{align*}
		\left|x_{0}-x_{1}\right|+4\rho & \leq r+4\left(\frac{\hat{t}-t_{1}}{cu(x_{1},t_{1})^{2-q}}\right)^{\frac{1}{q}}\leq r+4\left(\frac{\hat{t}-t_{0}}{cu(x_{1},t_{1})^{2-q}}\right)^{\frac{1}{q}}\\
		& =r+4\left(\frac{\hat{c}u(x_{0},t_{0})^{2-q}r^{q}}{cu(x_{1},t_{1})^{2-q}}\right)^{\frac{1}{q}}=r+4\left(\frac{\hat{c}}{c}\tilde{\mu}^{2-q}\right)^{\frac{1}{q}}r=5r.
		\end{align*}
		Further,
		\begin{align*}
		t_{1}+cu(x_{1},t_{1})^{2-q}(4\rho)^{q} & =t_{1}+cu(x_{1},t_{1})^{2-q}4^{q}\left(\frac{\hat{t}-t_{1}}{cu(x_{1},t_{1})^{2-q}}\right)\\
		& =t_{1}+4^{q}(\hat{t}-t_{1})\\
		& =(4^{q}-1)(t_{0}-t_{1})+t_{0}+4^{q}\hat{c}u(x_{0},t_{0})^{2-q}r^{q}\\
		& \leq t_{0}+\hat{c}u(x_{0},t_{0})^{2-q}(5r)^{q}.
		\end{align*}
		Thus the upper boundary of $(x_{1},t_{1})+Q_{4\rho}(\theta)$ is contained
		in $(x_{0},t_{0})+Q_{5r}(\hat{\theta})\subset Q_{1}^{-}(1)$.
	\end{proof}

	\section{Proof of the elliptic Harnack's inequality}
	To prove Theorem \ref{thm:ellipticharnack} we first establish the following version where more space is required around the point $(x_0, t_0)$. To prove this proposition, we first use the parabolic Harnack Theorem \ref{thm:backharnack} to get an estimate at an earlier time level, use Lemma \ref{lem:supersolution} to construct a super solution with infinite boundary values at this level and finally use the comparison principle Theorem \ref{thm:comp} to get an estimate at our original time level. We repeat this process again around a local minimum of $u$ to get the other side of the inequality.
	\begin{proposition}\label{prop:elliptic harnack}\sloppy
		Let $u \geq 0$ be a viscosity solution to \eqref{eq:rgnppar} in $Q_{1}^{-}(1)$ and the range condition \eqref{eq:range} holds. Fix $\left(x_{0}, t_{0}\right) \in Q_{1}^{-}(1)$. Then there exist ${\bar{\gamma}=\bar{\gamma}(n, p, q)}$, ${c}={c}(n, p, q)$ and ${\alpha=\alpha(n,p,q)\in(0,1)}$ such that
		\begin{equation}
		\label{eq:ellipticharnack}
		\bar{\gamma}^{-1}\sup_{B_r(x_0)}u(\cdot,t_0)\leq u(x_0,t_0)\leq\bar{\gamma}\inf_{B_{r}(x_0)}u(\cdot,t_0),
		\end{equation}
		whenever $(x_0,t_0)+Q_{\frac{13}{\alpha}r}(\theta)\subset Q_{1}^{-}(1)$ where
		\begin{equation*}
		\theta={c}u\left(x_{0}, t_{0}\right)^{2-q}.
		\end{equation*}
	\end{proposition}
	\begin{proof}\sloppy
		We can use parabolic Harnack (Theorem \ref{thm:backharnack}) for radius $2r$ to obtain constants ${\mu=\mu(n, p, q)}$ and ${c}={c}(n, p, q)$ such that
		\begin{equation}\label{eq:ellitic harnack 11}
		u(x,t_0-\theta (2r)^{q})\leq\sup_{B_{2r}(x_0)}u(\cdot,t_0-\theta (2r)^{q})\leq\mu u(x_0,t_0)
		\end{equation}
		for all $x\in B_{2r}(x_0)$, where $\theta={c}u\left(x_{0}, t_{0}\right)^{2-q}$. This is justified because $\frac{6}{\alpha}(2r)<\frac{13}{\alpha}r$.
		Let 
		\[
		v(x,t):=\lambda(t-t_{0}+\theta(2r)^{q})^{\frac{1}{2-q}}\left(\frac{1}{(2r)^{\frac{1}{1-q}}((2r)^{\frac{q}{q-1}}-\left|x-x_{0}\right|^{\frac{q}{q-1}})}\right)^{\frac{q}{2-q}}+\mu u(x_{0},t_{0}).
		\]
		Then by Lemma \ref{lem:supersolution}, $v$ is a viscosity supersolution in $B_{2r}(x_0)\times(t_{0}-\theta(2r)^{q},\infty)$
		that satisfies
		\[
		\begin{cases}
		v\geq\mu u(x_{0},t_{0}) & \text{on }B_{2r}(x_0)\times\left\{ t_{0}-\theta(2r)^{q}\right\} ,\\
		\lim\limits_{\Omega_{T} \ni(x,t)\to(y,s)}v(x,t)=\infty & \text{for all }(y,s)\in\partial B_{2r}(x_0)\times(t_{0}-\theta(2r)^{q},\infty)
		\end{cases}
		\]
		and we can use comparison principle Theorem \ref{thm:comp} to get
		\begin{equation}\label{eq:elliptic harnack 22}
		u\leq v \text{ in } (x_0,t_0)+Q_{2r}(\theta)
		\end{equation}
		because $u$ is bounded in $(x_0,t_0)+Q_{2r}(\theta)$ and on the bottom of the cylinder we have by \eqref{eq:ellitic harnack 11}
		\begin{equation*}
		u(x,t_0-\theta(2r)^{q})\leq\mu u(x_0,t_0)\leq v(x,t_0-\theta(2r)^{q}).
		\end{equation*}
		The estimate \eqref{eq:elliptic harnack 22} and the definition of $\theta$ imply in particular that
		\begin{align*}
		\label{eq:ellipticharnack1}
		\sup_{B_{r}(x_{0})}u(\cdot,t_{0})\leq\sup_{B_{r}(x_{0})}v(\cdot,t_{0}) & =\lambda\left(\theta (2r)^{q}\right)^{\frac{1}{2-q}}\left((2r)^{\frac{1}{1-q}}((2r)^{\frac{q}{q-1}}-r^{\frac{q}{q-1}})\right)^{\frac{q}{q-2}}+\mu u(x_{0},t_{0})\\
		& =\lambda(c u(x_{0},t_{0})^{2-q}2^{q}r^{q})^{\frac{1}{2-q}}\left(r2^{\frac{1}{1-q}}(2^{\frac{q}{q-1}}-1)\right)^{\frac{q}{q-2}}+\mu u(x_{0},t_{0})\\
		& =\left(\lambda c^{\frac{1}{2-q}}2^{\frac{q}{2-q}}\left(2^{\frac{1}{1-q}}(2^{\frac{q}{q-1}}-1)\right)^{\frac{q}{q-2}}+\mu\right)u(x_{0},t_{0})\\
		& =:\bar{\gamma}(n,p,q)u(x_{0},t_{0}).\numberthis
		\end{align*}
		Dividing by $\bar{\gamma}$ gives us the left side of \eqref{eq:ellipticharnack}. The constant $\bar{\gamma}$ blows up in the limit cases because $\lambda$ blows up when $q\to2$ for all $c$, and $\mu$ does the same when $q$ approaches the lower bound of \eqref{eq:range}. 
		
		Let $\hat{x}$ be a minimum point of $u(\cdot,t_{0})$ in $\overline{B}_{r}(x_{0})$. We will again use Theorem \ref{thm:backharnack} to obtain
		\begin{equation*}
		\sup_{B_{2r}(\hat{x})}u(\cdot,t_0-\hat{\theta}(2r)^{q})\leq\mu u(\hat{x},t_0)
		\end{equation*}
		where $\hat{\theta}={c}(u(\hat{x},t_{0}))^{2-q}$. The use of Harnack is justified because $\frac{6}{\alpha}(2r)+r<\frac{13}{\alpha}r$ because $\alpha\in(0,1)$. Let
		\[
		\hat{v}(x,t)=\lambda(t-t_{0}+\hat{\theta}(2r)^{q})^{\frac{1}{2-q}}\left(\frac{1}{(2r)^{\frac{1}{1-q}}((2r)^{\frac{q}{q-1}}-\left|x-\hat{x}\right|^{\frac{q}{q-1}})}\right)^{\frac{q}{2-q}}+\mu u(\hat{x},t_{0}).
		\]
		Then again by Lemma \ref{lem:supersolution}, $\hat{v}$ is a viscosity supersolution in $B_{2r}(\hat{x})\times(t_{0}-\hat{\theta}(2r)^{q},\infty)$
		that satisfies
		\[
		\begin{cases}
		\hat{v}\geq\mu u(\hat{x},t_{0}) & \text{on }B_{2r}(\hat{x})\times\{t_{0}-\hat{\theta}(2r)^{q}\},\\
		\lim\limits_{\Omega_{T} \ni(x,t)\to(y,s)}\hat{v}(x,t)=\infty & \text{for all }(y,s)\in\partial B_{2r}(\hat{x})\times(t_{0}-\hat{\theta}(2r)^{q},\infty)
		\end{cases}
		\]
		and we can use comparison principle Theorem \ref{thm:comp} to get
		\begin{equation*}
		u\leq \hat{v} \text{ in } (\hat{x},t_0)+Q_{2r}(\hat{\theta})
		\end{equation*}
		and thus
		\begin{align*}
		\label{eq:ellipticharnack2}
		u(x_{0},t_{0})\leq\sup_{B_{r}(\hat{x})}u(\cdot,t_{0})\leq\sup_{B_{r}(\hat{x})}\hat{v}(\cdot,t_{0}) & =\left(\lambda c^{\frac{1}{2-q}}2^{\frac{q}{2-q}}\left(2^{\frac{1}{1-q}}(2^{\frac{q}{q-1}}-1)\right)^{\frac{q}{q-2}}+\mu\right)u(\hat{x},t_{0})\\
		& =\bar{\gamma}(n,p,q)\inf_{B_{r}(x)}u(\cdot,t_{0}),\numberthis
		\end{align*}
		which is the right-hand side of \eqref{eq:ellipticharnack}.
		Combining \eqref{eq:ellipticharnack1} and \eqref{eq:ellipticharnack2} proves the theorem.
	\end{proof}
	
	Next we combine Proposition \ref{prop:elliptic harnack} with a covering argument to prove Theorem \ref{thm:ellipticharnack}. We first construct a suitable sequence of small balls along an arbitrary radial segment of our set. Then we show by induction that there is enough room around cylinders defined on these balls to use Proposition \ref{prop:elliptic harnack} to get an Harnack type estimate over any of these radial segments up arbitralily close to the boundary. Parabolic intrinsic Harnack chains for the $p$-parabolic equation have recently been examined in \cite{Avelin2019} in the degenerate case $p>2$.
	\begin{proof}[Proof of Theorem \ref{thm:ellipticharnack}] \sloppy
		By Proposition \ref{prop:elliptic harnack} there exist constants $\bar{\gamma}(n,p,q)$, $c^{\prime}(n,p,q)$ and ${\alpha(n,p,q)\in(0,1)}$ such that the elliptic Harnack's inequality
		\begin{equation}
		\bar{\gamma}^{-1}\sup_{B_{\tau}(z)}u(\cdot,t_{0})\leq u(z,t_{0})\leq\bar{\gamma}\inf_{B_{\tau}(z)}u(\cdot,t_{0})\label{eq:s elliptic harnack}
		\end{equation}
		holds whenever $B_{\frac{13}{\alpha}\tau}(z)\subset B_{1}$ and
		\begin{equation}
		t_{0}\pm\left(\frac{13}{\alpha}\tau\right)^{q}c^{\prime}u(z,t_{0})^{2-q}\in(-1,0].\label{eq:s second}
		\end{equation}
		
		Fix an arbitrary $\hat{y}\in\partial B_{r}(x_{0})$. Let $\rho:=r\alpha(\sigma-1)/13$.
		We define the points
		\[
		y_{k}:=x_{0}+k\rho\frac{\hat{y}-x_{0}}{\left|\hat{y}-x_{0}\right|}\in B_{r}(x_{0}),
		\]
		where $k=0,\ldots,K$ and $K\geq0$ is the smallest natural number
		such that $\hat{y}\in B_{\rho}(y_{K})$. Since $\hat{y}$ is on the
		boundary of $B_{r}(x_{0})$ and $\rho$ is a scaling of $r$, the
		number $K$ depends only on $\sigma$, $n$, $p$ and $q$. We will apply the elliptic
		Harnack's inequality in the balls $B_{\rho}(y_{k})$. Therefore we need
		the corresponding intrinsic cylinders to be contained within $Q_{1}^{-}(1)$.
		Since the choice of $\rho$ ensures that $B_{\frac{13}{\alpha}\rho}(y)\subset B_{\sigma r}(x_{0})\subset B_{1}$
		whenever $y\in B_{r}(x_{0})$, it remains to show that (\ref{eq:s second})
		holds for $\tau=\rho$ and $z=y_{k}$, $k=0,\ldots,K$. We choose
		\[
		c:=c^{\prime}\left(\frac{\sigma-1}{\sigma}\right)^{q}\bar{\gamma}^{K(2-q)}
		\]
		and proceed by induction to check that we have enough space in the time direction to use Proposition \ref{prop:elliptic harnack} for each of the cylinders $(y_k,t_0)+Q_{\rho}(c'u(y_k,t_0)^{2-q})$. Note that the assumption $(x_{0},t_{0})+Q_{\sigma r}(\theta)\subset Q_{1}^{-}(1)$
		implies
		\begin{equation}
		t_{0}\pm(\sigma r)^{q}cu(x_{0},t_{0})^{2-q}\in(-1,0].\label{eq:s first}
		\end{equation}
		
		\textbf{(Initial step)} Since $\bar{\gamma}\geq1$, we have
		\begin{align}
		\left(\frac{13}{\alpha}\rho\right)^{q}c^{\prime}u(y_{0},t_{0})^{2-q}=(r(\sigma-1))^{q}c^{\prime}u(x_{0},t_{0})^{2-q} & =(\sigma r)^{q}cu(x_{0},t_{0})^{2-q}\frac{c^{\prime}(\sigma-1)^{q}}{c\sigma^{q}}\nonumber \\
		& \leq(\sigma r)^{q}cu(x_{0},t_{0})^{2-q}.\label{eq:s 1}
		\end{align}
		It follows from (\ref{eq:s 1}) and (\ref{eq:s first}) that (\ref{eq:s second})
		holds with $z=y_{0}$ and $\tau=\rho$. Thus the elliptic Harnack
		inequality (\ref{eq:s elliptic harnack}) gives
		\[
		\bar{\gamma}^{-1}\sup_{B_{\rho}(y_{0})}u(\cdot,t_{0})\leq u(x_{0},t_{0})\leq\bar{\gamma}\inf_{B_{\rho}(y_{0})}u(\cdot,t_{0}).
		\]
		
		\textbf{(Induction step)} Suppose that $1\leq k\leq K$ and that we
		have
		\begin{equation}
		\bar{\gamma}^{-k}\sup_{B_{\rho}(y_{k-1})}u(\cdot,t_{0})\leq u(x_{0},t_{0})\leq\bar{\gamma}^{k}\inf_{B_{\rho}(y_{k-1})}u(\cdot,t_{0}).\label{eq:s i assumption}
		\end{equation}
		Since $y_{k}\in\overline{B}_{\rho}(y_{k-1})$, this implies in particular
		\[
		u(y_{k},t_0)\leq\bar{\gamma}^{k}u(x_{0},t_{0}).
		\]
		Therefore by definition of $\rho$ and $c$ we have
		\begin{align}
		\left(\frac{13}{\alpha}\rho\right)^{q}c^{\prime}u(y_{k},t_{0})^{2-q} & \leq(r(\sigma-1))^{q}c^{\prime}\bar{\gamma}^{k(2-q)}u(x_{0},t_{0})^{2-q}\nonumber \\
		& =(\sigma r)^{q}cu(x_{0},t_{0})^{2-q}\frac{c^{\prime}(\sigma-1)^{q}\bar{\gamma}^{k(2-q)}}{c\sigma^{q}}\nonumber \\
		& \leq(\sigma r)^{q}cu(x_{0},t_{0})^{2-q}.\label{eq:s i 1}
		\end{align}
		It follows from (\ref{eq:s i 1}) and (\ref{eq:s first}) that (\ref{eq:s second})
		holds for $z=y_{k}$ and $\tau=\rho$. Consequently by the elliptic
		Harnack's inequality (\ref{eq:s elliptic harnack}) we have
		\[
		\bar{\gamma}^{-1}\sup_{B_{\rho}(y_{k})}u(\cdot,t_{0})\leq u(y_{k},t_{0})\leq\bar{\gamma}\inf_{B_{\rho}(y_{k})}u(\cdot,t_{0}).
		\]
		Since $y_{k}\in\overline{B}_{\rho}(y_{k-1})$, combining the above
		display with (\ref{eq:s i assumption}) yields
		\[
		u(x_{0},t_{0})\geq\bar{\gamma}^{-k}\sup_{B_{\rho}(y_{k-1})}u(\cdot,t_{0})\geq\bar{\gamma}^{-k}u(y_{k},t_{0})\geq\bar{\gamma}^{-(k+1)}\sup_{B_{\rho}(y_{k})}u(\cdot,t_{0})
		\]
		and similarly
		\[
		u(x_{0},t_{0})\leq\bar{\gamma}^{k}\inf_{B_{\rho}(y_{k-1})}u(\cdot,t_{0})\leq\bar{\gamma}^{k}u(y_{k},t_{0})\leq\bar{\gamma}^{k+1}\inf_{B_{\rho}(y_{k})}u(\cdot,t_{0}).
		\]
		Thus
		\begin{equation}
		\bar{\gamma}^{-(k+1)}\sup_{B_{\rho}(y_{k})}u(\cdot,t_{0})\leq u(x_{0},t_{0})\leq\bar{\gamma}^{k+1}\inf_{B_{\rho}(y_{k})}u(\cdot,t_{0})\label{eq:s induction ineq}
		\end{equation}
		and the induction step is complete.
		
		By the induction principle, the estimate (\ref{eq:s induction ineq})
		holds for all $k=0,\ldots,K$. Since $\hat{y}\in B_{\rho}(y_{K})$,
		we have in particular
		\[
		\bar{\gamma}^{-(K+1)}\sup_{[x,\hat{y}]}u(\cdot,t_{0})\leq u(x_{0},t_{0})\leq\bar{\gamma}^{K+1}\inf_{[x,\hat{y}]}u(\cdot,t_{0}),
		\]
		where $[x,\hat{y}]$ denotes the segment from $x$ to $\hat{y}$.
		Since $\hat{y}\in\partial B_{r}(x_{0})$ was arbitrary, the estimate
		of the theorem follows for $\gamma:=\bar{\gamma}^{K+1}$. 
	\end{proof}
	

\end{document}